\documentclass[11pt]{amsart}
\usepackage{amsmath, amssymb, amsthm, amsfonts, mathrsfs}

\setbox0=\hbox{$+$}
\newdimen\plusheight
\plusheight=\ht0
\def\+{\;\lower\plusheight\hbox{$+$}\;}

\setbox0=\hbox{$-$}
\newdimen\minusheight
\minusheight=\ht0
\def\-{\;\lower\minusheight\hbox{$-$}\;}

\setbox0=\hbox{$\cdots$}
\newdimen\cdotsheight
\cdotsheight=\plusheight
\def\cds{\lower\cdotsheight\hbox{$\cdots$}}

\theoremstyle{definition}

\usepackage{graphicx, epsfig}
\theoremstyle{definition}
\makeatletter
\def\proof{\@ifnextchar[{\@oproof}{\@nproof}}
\def\@oproof[#1][#2]{\trivlist\item[\hskip\labelsep\textit{#2 Proof of Theorem 1.4.}~]\ignorespaces}
\def\@nproof{\trivlist\item[\hskip\labelsep\textit{Proof.}~]\ignorespaces}

\makeatother

\numberwithin{equation}{section}
\theoremstyle{plain}
\newtheorem{theorem}{Theorem}[section]
\newtheorem{corollary}[theorem]{Corollary}

\newtheorem{lemma}[theorem]{Lemma}
\newtheorem{claim}[theorem]{Claim}

\begin{document}

\title{Heegaard genera in congruence towers of hyperbolic 3-manifolds}
\author{BoGwang Jeon}\thanks{The author was partially supported by US NSF grant DMS-0707136}
\maketitle
\begin{abstract}
Given a closed hyperbolic 3-manifold $M$, we construct a tower of covers with increasing Heegaard genus and give an explicit lower bound on the Heegaard genus of such covers as a function of their degree. Using similar methods we prove that for any $\epsilon>0$ there exist infinitely many congruence covers $\{M_i\}$ such that, for any $x \in M$, $M_i$ contains an embbeded ball $B_x$ (with center $x$) satisfying $\text{vol}(B_x) > \left(\text{vol}(M_i)\right)^{\tfrac{1}{4}-\epsilon}$. We get the similar results for an arithmetic non-compact case.
\end{abstract}
\section{introduction}
 Let $M$ be a hyperbolic 3-manifold and $\{M_i\}$ be a collection of finite covers of $M$. The \emph{infimal Heegaard gradient} of $M$ with respect to $\{M_i\}$ is defined as:
 \begin{equation*}
 \inf_i \dfrac{\chi^h_-(M_i)}{[\pi_1(M):\pi_1(M_i)]},
 \end{equation*}
 where $\chi^h_-(M_i)$ denotes the minimal value for the negative of the Euler characteristic of a Heegaard surface in $M_i$.

 A fundamental question is whether the infimal Heegaard gradient is zero or not. This question is closely related to the potential solutions of several important conjectures in 3-manifold theory such as the virtual Haken conjecture and the virtual fibering conjecture \cite{lac}\cite{lac1}. Assuming the Lubotzky-Sarnak conjecture, a closed hyperbolic 3-manifold $M$ has a tower $\{M_i\}$ of finite covers without Property $\tau$. By a theorem of Lackenby \cite{lac} if the infimal Heegaard gradient of this tower is positive, then $M_i$ is Haken for sufficiently large $i$. According to a conjecture of Lackenby \cite{lac}, if the infimal Heegaard gradient of this tower is zero, then $M_i$ is fibered for some $i$. Thus the Heegaard gradient plays an important role in these approaches to the virtual Haken conjecture and the virtual fibering conjecture.

 When the manifold is \emph{arithmetic}, Lackenby proved that:
 \begin{theorem}\label{1.1} \cite{lac}
 Let $M$ be an arithmetic hyperbolic 3-manifold. Then there are positive constants $c$ and $C$ which depend only on $M$, such that for any congruence cover $M_i \rightarrow M$,
 \begin{equation*}
c~[\pi_1(M):\pi_1(M_i)]\leq \chi^h_-(M_i)\leq C~[\pi_1(M):\pi_1(M_i)].
 \end{equation*}
 \end{theorem}
He established this theorem by proving that Property $\tau$ with respect to a set of finite covers $\{M_i\}$ implies that $\{M_i\}$ has positive infimal Heegaard gradient. Since Luboltzky showed that an arithmetic hyperbolic 3-manifold has Property $\tau$ with respect to its congruence covers \cite{lu}, Theorem 1.1 follows.

As we've seen Heegaard genera and degrees of towers of covers provide important information and have strong connections with various things like Property $\tau$, but, unfortunately, little has been known about these in general \cite{lac}\cite{lo}. Here, we construct towers of finite covers of hyperbolic 3-manifolds with increasing Heegaard genera. While we do not show that the infimal Heegaard gradient is positive, we do give quantitative lower bounds for the Heegaard genus in terms of the degree of the cover. More precisely, we prove the following statement.

\begin{theorem}
Let $M$ be a closed hyperbolic 3-manifold and $\epsilon>0$ is an any (small) number. Then there exists a tower of finite congruence covers
\begin{equation*}
\dots \rightarrow M_i\rightarrow \dots \rightarrow M_2 \rightarrow M
\end{equation*}
with the Heegaard genus of $M_i$ $\geq [\pi_1(M):\pi_1(M_i)]^{\tfrac{1}{8}-\epsilon}$. If $M$ is arithmetic, then we can improve the exponent $\dfrac{1}{8}-\epsilon$ to $\dfrac{1}{4}-\epsilon$.
\end{theorem}

For the arithmetic non-compact case, we get a similar result to the above arithmetic closed case.

\begin{theorem}
For a given arithmetic non-compact hyperbolic 3-manifold $M$ and any $\epsilon>0$, there exists a tower of finite congruence covers
\begin{equation*}
\dots \rightarrow M_i\rightarrow \dots \rightarrow M_2 \rightarrow M
\end{equation*}
such that the Heegaard genus of $M_i$ $\geq [\pi_1(M):\pi_1(M_i)]^{\tfrac{1}{4}-\epsilon}$.
\end{theorem}

Although these results are weaker than Theorem 1.1 in the arithmetic case, our proofs involve different methods. In particular they use the result of Bachman, Cooper and White about the relation between the injectivity radius and the Heegaard genus of a hyperbolic 3-manifold (see Theorem 2.1 and Corollary 2.2). Later in Section 8, we will analyze the limitations of these methods. It will turn out that methods qualitatively similar to our own cannot prove analogues of Theorems 1.2 and 1.3 with $\dfrac{1}{8}-\epsilon$ and $\dfrac{1}{4}-\epsilon$ replaced by $x$ for any $x> \dfrac{1}{2}$. The proofs of Theorem 1.2 and Theorem 1.3 are similar in spirit but different in the details so we give them separately. In addition, we prove the following theorems using the tools in the proofs of Theorems 1.2 and 1.3; see Section 2 for the definition of the lower injectivity radius, the principal congruence subgroup and the Hecke-type congruences subgroups.

\begin{theorem}
For a given closed hyperbolic 3-manifold $M$ and for any $\epsilon>0$, there exists infinitely many congruence covers $\{M_i\}$ such that
\begin{equation*}
e^{r_i}> [\pi(M):\pi(M_i)]^{\tfrac{1}{8}-\epsilon}
\end{equation*}
holds for all $i$ where $r_i$ is the lower injectivity radius of $M_i$. In addition, for any $x \in M_i$, $M_i$ contains an embedded ball $B_x$ with center $x$ so that
\begin{equation*}
\emph{vol}(B_x) > \left(\emph{vol}(M_i)\right)^{\tfrac{1}{4}-\epsilon}
\end{equation*}
holds. If $M$ is arithmetic, then we can improve the exponents $\dfrac{1}{8}-\epsilon$ and $\dfrac{1}{4}-\epsilon$ to $\dfrac{1}{4}-\epsilon$ and $\dfrac{1}{2}-\epsilon$ respectively.
\end{theorem}

\begin{theorem}
 For a given arithmetic non-compact hyperbolic 3-manifold $M$, let $M'$ be a finite cover of $M$ such that its fundamental group $\Gamma'$ is a subgroup of a Bianchi group $PSL_2(\mathcal{O}_n)$ and let $I$ be a square-free ideal of $\mathcal{O}_n$ with no prime factors from a fixed finite set of prime ideals which depends only on $\Gamma'$. Then the following statements hold.\\
\\
\emph{(i)} For any $\epsilon>0$, there exists $d>0$ depending on $\epsilon$ and $\Gamma'$ such that if $M'_0(I)$ is a cover induced by a Hecke-type congruence subgroup $\Gamma'_0(I)$ with $[\Gamma':\Gamma'_0(I)]>d$, then $M'_0(I)$ contains an embedded ball $\text{B}$ which satisfies
\begin{equation*}
\emph{vol}(B) > \left(\emph{vol}(M'_0(I))\right)^{1/2-\epsilon}.
\end{equation*}
\emph{(ii)} There exists $d>0$ depending only on $\Gamma'$ such that if $M'_1(I)$ is a cover induced by a Hecke-type congruence subgroup $\Gamma'_1(I)$ with $[\Gamma':\Gamma'_1(I)]>d$, then $M'_1(I)$ contains an embedded ball $\text{B}$ which satisfies
\begin{equation*}
\emph{vol}(B) > c \left(\emph{vol}(M'_1(I))\right)^{1/4}.
\end{equation*}
\emph{(iii)} There exists $d>0$ depending only on $\Gamma'$ such that if $M'(I)$ is a cover induced by a principal congruence subgroup $\Gamma'(I)$ with $[\Gamma':\Gamma'(I)]>d$, then $M'(I)$ contains an embedded ball $\text{B}$ which satisfies
\begin{equation*}
\emph{vol}(B) > c \left(\emph{vol}(M'(I))\right)^{1/3}.
\end{equation*}
\end{theorem}

In fact, Theorem 1.5 (iii) is shown in \cite{yeung} with the better exponent $2/3$ in a different way.

Here is the outline of the paper. First, in Section 2, we review some basic facts which we use in the proofs. We prove Theorems 1.2 and 1.4 in Section 3-4 and 1.3 in Section 5-6. In Section 7, we shall show Theorem 1.4. Finally, we briefly analyze why our method falls short of proving Theorem 1.1 in Section 8.

\section{Some Background}

\subsection{Congruence Subgroups}
Let $\Gamma$ be the fundamental group of a finite-volume hyperbolic 3-manifold $M$ as a subgroup of $PSL_2(\mathbb{C})(\cong SL_2(\mathbb{C})/\{\pm I\})$. Then, after conjugating, we can assume that there exists an embedding
\begin{equation}\label{2.1}
\rho: \Gamma \hookrightarrow PSL_2(\mathcal{O}_S)
\end{equation}
where $\mathcal{O}_S$ is the $S$-$integers$ of an algebraic number field $K$ (see Theorem 3.2.8 in \cite{mr} taking $S$ to be the multiplicative set of the denominators of the generators of $\Gamma$). Given an ideal $J_S$ in $\mathcal{O}_S$, the \emph{principal congruence subgroup of level} $J_S$  of the group $\Gamma$ is the kernel of the natural reduction
\begin{equation}\label{2.1}
\rho_{J_S}: \Gamma \rightarrow PSL_2(\mathcal{O}_S/J_S)
\end{equation}
 and is denoted by $\Gamma(J_S)$. If $J_S=P_1...P_r$ is a square free ideal of $\mathcal{O}_S$ (so the $P_i$ are distinct prime ideals of $\mathcal{O}_S$), then by the Chinese Remainder Theorem, we have
\begin{equation*}
SL_2(\mathcal{O}_S/J_S)=SL_2(\mathcal{O}_S/P_1)\times...\times SL_2(\mathcal{O}_S/P_r).
\end{equation*}
Since
\begin{equation*}
|SL_2(\mathcal{O}_S/P_i)|=\text{N}{(P_i)}(\text{N}^2{(P_i)}-1)
\end{equation*}
for each prime ideal $P_i$ where $\text{N}{(P_i)}$ is the norm of an ideal $P_i$ in $\mathcal{O}_S$, we get
\begin{equation}\label{2.1}
|PSL_2(\mathcal{O}_S/J_S)|=\dfrac{1}{2}\prod^r_{i=1}\text{N}{(P_i)}(\text{N}^2{(P_i)}-1).
\end{equation}
Clearly the degree $[\Gamma:\Gamma(J_S)]$ is also bounded by the above number.

More generally, a \emph{congruence subgroup} of $\Gamma$ is a subgroup of $\Gamma$ which contains a principal congruence subgroup. Typical examples are the \emph{Hecke-type congruence subgroups} $\Gamma_0(J_S)$ and $\Gamma_1(J_S)$ which are defined to be

\begin{equation*}
\Gamma_0(J_S)=\left\{\gamma \in \Gamma~\Big|~\rho_{J_S}(\gamma)=
\left( {\begin{array}{cc}
 * & *  \\
 0 & *  \\
\end{array} } \right)\right\},
\end{equation*}
\begin{equation*}
\Gamma_1(J_S)=\left\{\gamma \in \Gamma~\Big|~\rho_{J_S}(\gamma)=
\left( {\begin{array}{cc}
 1 & *  \\
 0 & 1  \\
\end{array} } \right)\right\}
\end{equation*}
where $\left( {\begin{array}{cc}
 * & *  \\
 0 & *  \\
\end{array} } \right)$ and $\left( {\begin{array}{cc}
 1 & *  \\
 0 & 1  \\
\end{array} } \right)$ are the matrix representations of the elements in $PSL_2(\mathcal{O}_S/J_S)$. These groups can also be expressed as follow in more explicit forms;
\begin{equation*}
\Gamma_0(J_S)=\left\{\left. \left( {\begin{array}{cc}
 a & b  \\
 c & d  \\
\end{array} } \right) \in \hat{\Gamma}~\right|~\left( {\begin{array}{cc}
 a & b  \\
 c & d  \\
\end{array} } \right) \equiv
\left( {\begin{array}{cc}
 * & *  \\
 0 & *  \\
\end{array} } \right) \mod J_S   \right\}/\{\pm I\},
\end{equation*}
\begin{equation*}
\Gamma_1(J_S)=\left\{\left. \left( {\begin{array}{cc}
 a & b  \\
 c & d  \\
\end{array} } \right) \in \hat{\Gamma}~\right|~\left( {\begin{array}{cc}
 a & b  \\
 c & d  \\
\end{array} } \right) \equiv \pm
\left( {\begin{array}{cc}
 1 & *  \\
 0 & 1  \\
\end{array} } \right) \mod I_S   \right\}/\{\pm I\}
\end{equation*}
where $\hat{\Gamma}$ is the inverse image of $\Gamma$ in $SL_2(O_S)$.

Now we look at two simple cases where the map in (2.2) is surjective. First, for a prime ideal $P$ of $\mathcal{O}_S$, extend the map in (2.1) to
\begin{equation*}
\Gamma \hookrightarrow PSL_2(K_P)
\end{equation*}
where $K_P$ is the $P$-adic local field. Then this restricts to a map
\begin{equation}\label{2.5}
\Gamma \hookrightarrow PSL_2(\mathcal{O}_P)
\end{equation}
where $\mathcal{O}_P$ is the unique $p$-adic integers of $K_P$. If we consider the reduction map
\begin{equation}
\Gamma \rightarrow PSL_2(\mathcal{O}_P/\pi\mathcal{O}_P)
\end{equation}
of (2.4) where $\pi\mathcal{O}_P$ is the unique maximal ideal of $\mathcal{O}_P$, then it is clear that the map in (2.5) is actually the same as the one in (2.2) when $J_S=P$. According to \cite{lr}, this map in (2.5) is  surjective for almost all prime ideals $P$ such that $P$ is a prime ideal factor of a rational prime that splits completely in $\mathcal{O}_K$. A second case where (2.2) is surjective comes when $\Gamma$ is a subgroup of a Bianchi group; that is, $\Gamma \subset PSL_2(\mathcal{O}_K)$ where $\mathcal{O}_K$ is the ring of integers of an imaginary quadratic number field $K$. Under this assumption, by the Strong Approximation Theorem \cite{w}, $\Gamma$ is dense in $PSL_2(\mathcal{O}_P)$ for almost all prime ideals $P$. If we define a natural map
\begin{equation}
\phi: PSL_2(\mathcal{O}_P)\rightarrow PSL_2(\mathcal{O}_P/\pi\mathcal{O}_P),
\end{equation}
then, using the fact that $\Gamma$ is dense in $PSL_2(\mathcal{O}_P)$, we can get the following surjection
\begin{equation}
\Gamma \rightarrow PSL_2(\mathcal{O}_P)/\text{ker}~\phi.
\end{equation}
As $\mathcal{O}_P/\pi\mathcal{O}_P$ is isomorphic to $\mathcal{O}_K/P\mathcal{O}_K$, (2.6) and (2.7) give the surjective map
\begin{equation}
\Gamma \rightarrow PSL_2(\mathcal{O}_K/P\mathcal{O}_K).
\end{equation}

The above examples are particularly important because it is possible to calculate the indices of the various congruence subgroups explicitly in these cases. For example, if $I_S=P_1...P_r$ is a square free ideal of $\mathcal{O}_S$ such that the maps $\Gamma \rightarrow PSL_2(\mathcal{O}_S/P_i)$ are surjective for all prime ideals $P_i$, then the index of $\Gamma/\Gamma(I_S)$ is given by (2.3). Furthermore, under the same assumption, it can be shown that the indices $[\Gamma:\Gamma_0(I)]$ and $[\Gamma:\Gamma_1(I)]$ are equal to
\begin{equation}\label{2.2}
\prod^r_1 (\text{N}{(P_i)}+1),
\end{equation}
\begin{equation}\label{2.3}
\dfrac{1}{2}\prod^r_{i=1}(\text{N}^2{(P_i)}-1)
\end{equation}
respectively (see Chapter 4 in \cite{mi} for details).

\subsection{Injectivity Radius}
The \emph{injectivity radius} of a Riemannian manifold $M$ at a point $x\in M$, $\text{inj}_x(M)$, is the largest radius for which the exponential map at $x$ is a diffeomorphism. The upper injectivity radius, $\overline{\text{inj}}(M)$, is the supremum of $\text{inj}_x(M)$ as $x$ varies over $M$, and the lower injectivity radius, $\underline{\text{inj}}(M)$, is the infimum of $\text{inj}_x(M)$ as $x$ varies over $M$. In particular, when $M$ is hyperbolic, the upper injectivity radius of $M$ is equal to
\begin{equation}\label{2.8}
\frac{1}{2}\text{sup}\left\{\right.\text{min}_{g\neq I, g\in \Gamma} (d_{\mathbb{H}^3}(x,g(x)))~\left|~x\in \mathbb{H}^3 \right\}
\end{equation}
where $\Gamma$ is the fundamental group of $M$. Moreover if $M$ is closed, then the lower injectivity radius of $M$ has the same value as half of the shortest length of a closed geodesic of $M$.

Bachmann, Cooper and White proved the following theorem which provides an important method for bounding the Heegaard genus in terms of the injectivity radius.
\\
\begin{theorem}\cite{bac}
If $M$ is a closed hyperbolic 3-manifold and $r=\overline{\text{inj}}(M)$, then
\begin{align*}
\emph{Heegaard genus of}~M \geq \dfrac{1}{2}\cosh r.
\end{align*}
\end{theorem}

Although the above theorem was proved for closed manifolds, using Dehn filling we can extend the theorem as follows (see \cite{ru} for a similar result).

\begin{corollary}
The above inequality holds for finite-volume non-compact hyperbolic 3-manifolds.
\end{corollary}
\begin{proof}
 A finite-volume non-compact hyperbolic 3-manifold $M$ can be approximated as a geometric limit of closed manifolds, that is, $M=\lim {M_n}$ where $\{M_n\}$ are closed manifolds obtained by Dehn filling. If we define $g$ (resp. $g_n$) to be the Heegaard genus of $M$ (resp. $M_n$), then the inequality $g\geq g_n$ is true for all $n$ because Dehn filling never increases the Heegaard genus. Let $r$ (resp. $r_n$) be the upper injectivity radius of $M$ (resp. $M_n$). Then there exists an $\epsilon>0$ such that the $\epsilon$-thick part $M_{[\epsilon,\infty)}$ of $M$ has the upper injectivity radius the same as $M$. Moreover, we can choose a uniform $\epsilon>0$ so that this is true for all $M_n$. Because $M_{n[\epsilon,\infty)}$ is approximately isometric to $M_{[\epsilon,\infty)}$ as $n\rightarrow \infty$, we get $r=\lim_{n \rightarrow \infty} r_n$. Now Corollary 2.2 follows from this, $g\geq g_n$ and Theorem 2.1.
\end{proof}
$\quad$\\

The above theorem and corollary will be applied to calculate lower bounds for Heegaard genera as we mentioned in Section 1. Specifically we use $\underline{\text{inj}}(M)$ in the proof of Theorem 1.2 and $\overline{\text{inj}}(M)$ in the proof of Theorem 1.3.

\subsection{Closed Geodesics}
Next we quickly review closed geodesics. A closed geodesic of a hyperbolic 3-manifold is always induced by a hyperbolic element of its fundamental group as an invariant axis. We can detect its length by the trace value of corresponding hyperbolic element (Chapter 11 in \cite{mr}). Concerning the asymptotic number of closed geodesics of a given closed hyperbolic 3-manifold $M$ as a function of length, we have the following nice formula:
\\
\\
\textbf{Prime Geodesic Theorem} \cite{mar}
\emph{For a closed hyperbolic 3-manifold $M$, the number of primitive elements of length less than or equal to $l$ is asymptotic to $e^{2 l}/{2 l}$ as $l$ goes to infinity.
}
\\

Here a \emph{primitive element} of $\Gamma$ is one which is not a nontrivial power of any element in $\Gamma$. If we denote $\#(l)$ the number of closed geodesics of length less than equal to $l$ in $M$, then we can get the upper bound of $\#(l)$ using the Prime Geodesic Theorem.

\begin{corollary}
Suppose that $M$, $\Gamma$ and $\#(l)$ are the same as above. Then there exists a constant $c'$ depending only on $\Gamma$ such that
\begin{equation*}
\#(l)< c'e^{2l}.
\end{equation*}
\end{corollary}
\begin{proof}
Put $\#_{prm}(l)$ the number of primitive elements of length less than or equal to $l$. Then, by the Prime Geodesic Theorem, there exists $d>0$ such $\#_{prm}(l)<\dfrac{e^{2l}}{l}$ for all $l>d$. Since every hyperbolic element $h \in \Gamma$ is of the form $g^m$ where $g$ is a primitive element in $\Gamma$ and $m \in \mathbb{N}$, we have
\begin{equation}
\#(l)=\sum_{i=1}^{\infty}\#_{prm}(l/i).
\end{equation}
Let $s$ be the length of a shortest geodesic of $M$. Because $\#_{prm}(l/i)=0$ for $i>\lfloor l/s\rfloor$, we can rewrite (2.12) as
\begin{equation}
\#(l)=\sum_{i=1}^{\infty}\#_{prm}(l/i)=\sum_{i=1}^{\lfloor l/s \rfloor}\#_{prm}(l/i).
\end{equation}
Clearly $\#_{prm}(l/i)\leq \#_{prm}(l)$. Thus, if $l>d$, then $\#_{prm}(l/i)<\dfrac{e^{2l}}{l}$ for $1 \leq i \leq \lfloor l/s \rfloor$. Combining this with (2.13), we get
\begin{equation*}
\#(l)<\lfloor l/s \rfloor \dfrac{e^{2l}}{l} \leq \dfrac{e^{2l}}{s}
\end{equation*}
for $l>d$. If we take $c'>0$ bigger than $\#(d)$ and $1/s$, then $\#(l)< c'e^{2l}$ for all $l>0$.
\end{proof}
$\quad$\\

\subsection{Number Theory}
Lastly we quote two important theorems from number theory and deduce two corollaries of them.
\\
\\
\textbf{Prime Number Theorem} \cite{bate}
\emph{Let $\pi(x)$ be the number of rational primes which are less than or equal to $x$. Then $\pi(x)$ is asymptotic to $\dfrac{x}{\log x}$ as $x$ goes to infinity and we denote this by
\begin{equation*}
\pi(x)\sim \dfrac{x}{\log x}.
\end{equation*}
In addition, this is equivalent to
\begin{equation*}
\theta(x)\sim x.
\end{equation*}
where $\theta(x)=\sum_{p<x}\log p$}.
\\
\\
\textbf{Chebotarev's Density Theorem} \cite{nir}
\emph{Let $K$/$\mathbb{Q}$ be a number field and $L$ be the Galois closure of $K$. If $S_1$ denotes the set of all primes of $\mathbb{Z}$ which split completely over $K$, then the following inequality holds.
\begin{equation*}
 \liminf_{x \rightarrow \infty} \dfrac{\#\{p\in S_1~|~p \leq x\}}{\#\{p~|~p \leq x\}} \geq \dfrac{1}{n}
\end{equation*}
where $n=[L:\mathbb{Q}]$.}

\begin{corollary}
With the same notations as in the above theorem, we can find a subset $S_2$ of $S_1$ such that
\begin{equation*}
 \lim_{x \rightarrow \infty} \dfrac{\#\{p\in S_2~|~p \leq x\}}{\#\{p~|~p \leq x\}}=\dfrac{1}{n},
\end{equation*}
i.e.,
\begin{equation*}
\lim_{x \rightarrow \infty} \dfrac{\pi_{S_2}(x)}{x/\log x}= \dfrac{1}{n}
\end{equation*}
where $\pi_{S_2}(x)$ is the number of primes of $S_2$ which are less than or equal to $x$.
Furthermore, the above formulas are equivalent to
\begin{equation*}
\lim_{x \rightarrow \infty} \dfrac{\theta_{S_2}(x)}{x}= \dfrac{1}{n}
\end{equation*}
where $\theta_{S_2}(x)=\sum_{p<x,~p\in S_2}\log p$.
\end{corollary}
\begin{proof}
 From $\liminf_{x \rightarrow \infty} \dfrac{\#\{p\in S_1~|~p \leq x\}}{\#\{p~|~p \leq x\}} \geq \dfrac{1}{n}$, the first result follows. The second argument can be deduced by copying the analogous steps of the proof of the equivalence of $\pi(x)\sim \dfrac{x}{\log x}$ and $\theta(x)\sim x$ (for example, see Chapter 4 in \cite{bate}).
\end{proof}
\begin{corollary}
With the same notations as in Corollary 2.4, let $p_k$ be the k-th prime number in $S_2$ and $d_k=p_1...p_k$. Then, for sufficiently large $k$, $p_{k+1}$ is less than $2n \log d_k$ and so, for any sufficiently large natural number $x$, there exists a prime number $p\in S_2$ such that $p\nmid x$ and $p< 2n \log x$.
\end{corollary}
\begin{proof}
This immediately follows from the definition of $\theta_{S_2}(x)$ and the formula $\lim_{x \rightarrow \infty} \dfrac{\theta_{S_2}(x)}{x}= \dfrac{1}{n}$.
\end{proof}

\section{Proof of Theorems 1.2 and 1.4}
In this section we prove Theorems 1.2 and 1.4. Throughout Sections 3 and 4, $\Gamma$ is the fundamental group of the closed hyperbolic 3-manifold $M$ and $\hat{\Gamma}$ is the inverse image of $\Gamma$ in $SL_2(O_S)$. We also denote the two inverse images of $\gamma \in \Gamma$ in $\hat{\Gamma}$ by $\pm\hat{\gamma}$. \\
\\
\emph{Proof of Theorem 1.2} We start by sketching the key idea of the proof. For a given closed geodesic of length less than or equal to $l$, using the facts that a closed geodesic is always induced by a hyperbolic element and that $\#(l)$ is finite, we find a prime ideal $P$ (of $\mathcal{O}_S$) such that its principal congruence group $\Gamma(P)$ doesn't contain any hyperbolic elements of length less than or equal to $l$. Then, applying Theorem 2.1 and formula (2.10), we calculate bounds for the Heegaard genus and the index of $\Gamma(P)$. The next lemma, which we'll prove in next section, is important for calculating these bounds.

\begin{lemma}
For $\omega \in \Gamma$ of hyperbolic length less than or equal to $l$, there exists $\alpha, \beta \in \mathcal{O}_K$ such that $\pm \emph{tr} \hat{\omega} =\pm \alpha/\beta$ and
\begin{equation*}
|\emph{N}(\alpha\pm 2\beta)|<(C_3)^{l}
\end{equation*}
where $C_3\geq 1$ is a constant which depends only on $\Gamma$.
\end{lemma}
The proof of the lemma is not difficult but it involves some preliminaries. So we'll prove it independently in Section 4. Here, using the lemma, we prove the following claim.\\

\begin{lemma}
For any $d>0$, there exists a Hecke-type congruence subgroup $\Gamma_1(P_S)$ such that $[\Gamma:\Gamma_1(P_S)]>d$ and
\begin{equation}\label{3.1}
\text{Heegaard genus of }M_1(P_S) \geq [\Gamma:\Gamma_1(P_S)]^{\tfrac{1}{8}-\tfrac{\epsilon}{2}}
\end{equation}
where $M_1(P_S)$ is the cover induced by $\Gamma_1(P_S)$ and $\epsilon>0$ is an any small number.
\end{lemma}
\begin{proof}
First, let $l>0$ be an arbitrary number and
\begin{equation*}
\{\pm\text{tr}\hat{\omega}_{1},\pm\text{tr}\hat{\omega}_{2},...,\pm \text{tr}\hat{\omega}_{{r(l)}}\}
\end{equation*}
be the set of traces of images in $\hat{\Gamma}$ of all hyperbolic elements of length less than or equal to $l$. Then, by Corollary 2.3, $r(l)\leq \#(l) < c'e^{2l}$ for some constant $c'$ which depends only on $\Gamma$. Now Lemma 3.1 implies that, for each $i$, we can find $\alpha_i, \beta_i \in \mathcal{O}_K$ with $\pm \text{tr}\hat{\omega_i}=\pm \alpha_i/\beta_i$ such that
\begin{equation*}
\prod_{i=1}^{r(l)}{|\text{N}{(\alpha_i-2\beta_i)}|}{|\text{N}{(\alpha_i+2\beta_i)}|} \leq (C_3^{2l})^{r(l)} \leq (C_3^{2l})^{c'e^{2 l}}.
\end{equation*}
\begin{claim}
If $l\rightarrow \infty$, then $r(l)\rightarrow \infty$ and $\prod_{i=1}^{r(l)}{|\text{N}{(\alpha_i-2\beta_i)}|}{|\text{N}{(\alpha_i+2\beta_i)}|}\rightarrow \infty$.
\end{claim}
\begin{proof}
Suppose $\prod_{i=1}^{r(l)}{|\text{N}{(\alpha_i-2\beta_i)}|}{|\text{N}{(\alpha_i+2\beta_i)}|}$ is bounded as $l\rightarrow \infty$. Pick a rational prime $p$ such that $p$ doesn't divide the norm of any generators of $S$ and $p>\prod_{i=1}^{r(l)}{|\text{N}{(\alpha_i-2\beta_i)}|}{|\text{N}{(\alpha_i+2\beta_i)}|}$ for all $l$. If $P$ is a prime factor of $p\mathcal{O}_K$, then
\begin{equation}
\text{N}(P)\nmid \prod_{i=1}^{r(l)}{|\text{N}{(\alpha_i-2\beta_i)}|}{|\text{N}{(\alpha_i+2\beta_i)}|}
\end{equation}
and so, for all $i$,
\begin{equation}
\alpha_i\pm2\beta_i\notin P \Rightarrow (\alpha_i\pm 2\beta_i)/\beta_i=\text{tr}\hat{\omega}_i\pm 2\notin P_S \Rightarrow \omega_i\notin \Gamma_0(P_S)
\end{equation}
where $\Gamma_0(P_S)$ is a Hecke-type congruence subgroup of $\Gamma$. This implies $\Gamma_0(P_S)$ doesn't contain any elements of $\Gamma$, contradicting to the fact that $\Gamma_0(P_S)$ is a finite index subgroup of $\Gamma$.
\end{proof}

By the above claim and Corollary 2.5, if $l$ is sufficiently large, then there exists a prime $p$ such that $p\nmid \prod_{i=1}^{r(l)}|{\text{N}{(\alpha_i-2\beta_i)}}|{|\text{N}{(\alpha_i+2\beta_i)}|}$, $p$ splits completely over $K$ and
\begin{equation*}
p<2n\log \left(\prod_{i=1}^{r(l)}{|\text{N}{(\alpha_i-2\beta_i)}|}{|\text{N}{(\alpha_i+	2\beta_i)}|}\right)<2n\log (C_3^{2c'le^{2 l}})
\end{equation*}
where $n$ is the degree of the Galois closure of $K$ as we previously defined in Section 2.4. Define $p(l)$ to be the smallest prime which satisfies the above conditions for the given $l$. Then, by the same reasoning as in the proof of Claim 3.3, we have $p(l)\rightarrow \infty$ as $l\rightarrow \infty$.

Now let's assume $l$ is sufficiently large so that $p(l)$ doesn't divide the norm of any generators of $S$ and any prime factor of $p(l)\mathcal{O}_S$ gives a surjection in (2.2). Set $P_S$ to be one of the prime factors of $p(l)\mathcal{O}_S$, $\Gamma_1(P_S)=$ Hecke-type congruence subgroup and $M_1(P_S)=$ corresponding cover of $\Gamma_1(P_S)$. Since $P_S$ satisfies (3.2) and (3.3), the length of a shortest closed geodesic of $\Gamma_1(P_S)$ is bigger than $l$ and so
\begin{equation*}
\text{Heegaard genus of }M_1(P_S)\geq \dfrac{1}{4}e^{l/2}.
\end{equation*}
As $\text{N}(P_S)$ is equal to $p(l)$, the degree $[\Gamma:\Gamma_1(P_S)]$ is less than
\begin{equation*}
\dfrac{1}{2}(p(l))^2
\end{equation*}
by (2.10) and, thus,
\begin{equation*}
\dfrac{1}{2}(p(l))^2<\dfrac{1}{2}(2n\log (C_3^{2c'le^{2 l}}))^{2}=2n^2\log^{2}(C_3^{2c'le^{2 l}})=
8n^2(c'le^{2 l})^{2}\log^{2}C_3
\end{equation*}
because $p(l)<2n\log (C_3^{2c'le^{2l}})$. Now it is easy to check that, for any $\epsilon>0$,
\begin{equation*}
\dfrac{e^{l/2}}{4}\geq \left(8n^2(c'le^{2 l})^{2}\log^{2}C_3\right)^{\tfrac{1}{8}-\tfrac{\epsilon}{2}}
\end{equation*}
for sufficiently large $l$. This means
\begin{equation*}
\text{Heegaard genus of }M_1(P_S) \geq [\Gamma:\Gamma_1(P_S)]^{\tfrac{1}{8}-\tfrac{\epsilon}{2}}.
\end{equation*}
From the construction, it is clear that we can make $M_1(P_S)$ with arbitrary large degree and Heegaard genus.
\end{proof}
$\quad$\\
Note that, in the proof of Lemma 3.2, we actually showed
\begin{equation}
 \text{Heegaard genus of }M_1(P_S)\geq \dfrac{e^{t/2}}{4} \geq\left(\dfrac{1}{2}(\text{N}(P_S))^{2}\right)^{\tfrac{1}{8}-\tfrac{\epsilon}{2}}
\end{equation}
where $t$ is the length of a shortest closed geodesic in $M_1(P_S)$.\\

Now we go back to the proof of Theorem 1.2 and let's construct a tower of finite covers of $M$. First, consider a Hecke congruence subgroup $\Gamma_1(P_1)$ of a prime ideal $P_1$ of $\mathcal{O}_S$ which satisfies the inequality (3.1) for a given $\epsilon>0$. Next pick a prime ideal $P_2$ (from Lemma 3.2) with $\text{N}(P_2)$ sufficiently large such that it satisfies
\begin{equation}
  \left(\dfrac{1}{2}(\text{N}(P_2))^{2}\right)^{\tfrac{1}{8}-\tfrac{\epsilon}{2}}>\left(\dfrac{1}{2}(\text{N}(P_1P_2))^{2}\right)^{\tfrac{1}{8}-{\epsilon}}.
\end{equation}
If $M(P_1P_2)$ (resp. $M(P_2)$) denotes a corresponding manifold of a Hecke-type congruence subgroup $\Gamma_1(P_1P_2)$ (resp. $\Gamma_1(P_2)$), then the length of a shortest closed geodesic in $M_1(P_1P_2)$ is bigger than the length of a shortest geodesic in $M(P_2)$ because $\Gamma_1(P_1P_2) \subset \Gamma_1(P_2)$. Thus, by (3.4), the Heegaard genus of $M_1(P_1P_2)$ is at least $\left(\dfrac{1}{2}(\text{N}(P_2))^{2}\right)^{\tfrac{1}{8}-\tfrac{\epsilon}{2}}$. Since the degree of $\Gamma_1(P_1P_2)$ is less than $\dfrac{1}{2}(\text{N}(P_1P_2))^{2}$, from (3.5), we get
\begin{equation*}
\text{Heegaard genus of }M_1(P_1P_2)\geq[\Gamma:\Gamma_1(P_1P_2)]^{\tfrac{1}{8}-\epsilon}.
\end{equation*}
By induction, for $n \geq 2$, let's pick a prime ideal $P_{n+1}$ having sufficiently large $\text{N}(P_{n+1})$ so that it satisfies
\begin{equation*}
  \left(\dfrac{1}{2}(\text{N}(P_{n+1}))^{2}\right)^{\tfrac{1}{8}-\tfrac{\epsilon}{2}}>\left(\dfrac{1}{2}(\text{N}(P_1...P_{n+1}))^{2}\right)^{\tfrac{1}{8}-{\epsilon}}.
\end{equation*}
 Define $M_1(P_1...P_{n+1})$ to be the corresponding cover of the Hecke-type congruence subgroup $\Gamma_1(P_1...P_{n+1})$. Then, by (2.10), the degree of $\Gamma_1(P_1...P_{n+1})$ is less than $\dfrac{1}{2}(\text{N}(P_1...P_n))^{2}$ but the Heegaard genus of $M_1(P_1...P_{n+1})$ is at least $\left(\dfrac{1}{2}(\text{N}(P_{n+1}))^{2}\right)^{\tfrac{1}{8}-\tfrac{\epsilon}{2}}$ by the same reasoning we explained above. Hence,
\begin{equation*}
\text{Heegaard genus of }M_1(P_1...P_{n+1})\geq [\Gamma:\Gamma_1(P_1...P_{n+1}))]^{\tfrac{1}{8}-\epsilon}.
\end{equation*}
$\Gamma_1(P_1),~\Gamma_1(P_1P_2)...$ gives a desired sequence of congruence covers for Theorem 1.2. This completes the proof of Theorem 1.2 for the general case.\\

(\emph{Arithmetic case}) If $M$ is arithmetic, then there exists a cover $M'$ with fundamental group $\Gamma'$ such that $\Gamma'$ is a subgroup of a maximal order of a quaternion algebra (see Chapter 8 in \cite{mr}). In this case, it is proved in \cite{mark} that the number of distinct complex lengths of real length less than or equal to $l$ in $M'$ is bounded by $c'' e^l$ where $c''$ is a constant depending only on $M'$. Applying the bound $c'' e^l$ instead of $c'e^{2l}$, we can check that Lemma 3.2 and the above construction of sequence of covers still hold for $M'$ with the reduced exponents $\dfrac{1}{4}-\dfrac{\epsilon}{2}$ and $\dfrac{1}{4}-\epsilon$ instead of $\dfrac{1}{8}-\dfrac{\epsilon}{2}$, $\dfrac{1}{8}-\epsilon$ respectively. In other words, for any $\epsilon>0$, $M'$ has a sequence of congruence covers $\{\Gamma_i\}$ such that
\begin{equation*}
\text{Heegaard genus of the cover induced by }\Gamma_i \geq [\Gamma:\Gamma_i]^{\tfrac{1}{4}-\epsilon}
\end{equation*}
with arbitrary large $[\Gamma:\Gamma_i]$ for each $i$.

Using this, we show $M$ has the same property. For any $\epsilon>0$, first pick $\epsilon'>0$ which is smaller than $\epsilon$ and construct a tower of Hecke-type congruence subgroups $\{\Gamma'_i\}$ of $\Gamma'$ that satisfies
\begin{equation*}
[\Gamma':\Gamma'_i]^{\tfrac{1}{4}-\epsilon'}\geq [\Gamma:\Gamma'_i]^{\tfrac{1}{4}-\epsilon}=([\Gamma:\Gamma'][\Gamma':\Gamma'_i])^{\tfrac{1}{4}-\epsilon}
\end{equation*}
and
\begin{equation*}
\text{Heegaard genus of the cover induced by }\Gamma'_i \geq ([\Gamma':\Gamma'_i])^{\tfrac{1}{4}-\epsilon'}.
\end{equation*}
for each $i$. The existence of $\{\Gamma'_i\}$ is guaranteed by the earlier discussion. Clearly $\{\Gamma'_i\}$ satisfies
\begin{equation*}
\text{Heegaard genus of the cover induced by }\Gamma'_i \geq ([\Gamma:\Gamma'_i])^{\tfrac{1}{4}-\epsilon}.
\end{equation*}
This gives a desired tower of finite covers of $M$ for the given $\epsilon$.\\
\\
\emph{Proof of Theorem 1.4} Let $\epsilon>0$ be fixed. In the proof of Lemma 3.2, we used the fact that
\begin{equation*}
\dfrac{e^{l/2}}{4}\geq \left(8n^2(c'le^{2 l})^{2}\log^{2}C_3\right)^{\tfrac{1}{8}-\tfrac{\epsilon}{2}}
\end{equation*}
holds for sufficiently large $l$. Changing the inequality slightly so that, for any given $0<c_1<\pi/2$, if we write
\begin{equation}
e^{l/2}\geq \left(8n^2(c'le^{2 l})^{2}\log^{2}C_3\right)^{\tfrac{1}{8}-\tfrac{\epsilon}{2}} \left(\geq [\Gamma:\Gamma_1(P_S)]^{\tfrac{1}{8}-\tfrac{\epsilon}{2}}\right)
\end{equation}
and
\begin{equation}
c_1 e^{l}\geq \left(\text{vol}(M)~8n^2(c'le^{2 l})^{2}\log^{2}C_3\right)^{\tfrac{1}{4}-\epsilon}\left(\geq \text{vol}(M(P_S))^{\tfrac{1}{4}-\epsilon}\right),
\end{equation}
then the above inequalities are also true for all sufficiently large $l$. The volume of hyperbolic ball of radius $r$ is
\begin{equation*}
\pi \left(\sinh (2r)-2r\right)
\end{equation*}
and it is bounded below by $c_1 e^{2r}$ when $r$ is sufficiently large. Now the result follows from this, (3.6), (3.7), and the similar steps in the proof of Lemma 3.2. If $M$ is arithmetic, then we can get the desired one by replacing the exponents $\dfrac{1}{8}-\dfrac{\epsilon}{2}$ in (3.6) and $\dfrac{1}{4}-\epsilon$ in (3.7) with $\dfrac{1}{4}-\dfrac{\epsilon}{2}$, $\dfrac{1}{2}-\epsilon$ respectively.

\section{Proof of Lemma 3.1}

First we introduce definitions and some preliminaries which are necessary to develop the proof of Lemma 3.1. Let $R$ be a finite set of generators of $\Gamma$. The \emph{minimal word length} of $\omega \in \Gamma$ is defined to be
\begin{equation*}
\text{min} \{|t_{1}|+...+|t_{k}|~|~\omega=\gamma_{1}^{t_{1}}...\gamma_{k}^{t_{k}}~\text{and} ~\gamma_{i}\in R~\text{for}~1\leq i \leq k\}.
\end{equation*}
According to \cite{mir}, since $M$ is compact, the Cayley graph of $\Gamma$ with respect to $R$ is quasi-isometric to its universal cover $\mathbb{H}^3$ and so the minimal word length of $\omega$ is bounded by $c'l$ where $l$ is the length of $\omega$ and $c'$ is a positive constant which depends only on $\Gamma$.

Let $G$ be the finite set of Galois embeddings of the number field $K$ in $\mathbb{C}$ and
\begin{equation*}
 C_1=\text{max}\Big\{1,~\text{max}\left\{\left.|\sigma(a)|~\right|~a~\text{is an entry of}~\pm \hat{\gamma}~\text{for any}~\gamma \in R,~\sigma \in G \right\}\Big\}.
\end{equation*}
Note that for any $\gamma \in R$ the constant $C_1$ is also an upper bound for the absolute values of all Galois conjugates of all the entries of $\pm \hat{\gamma}^{{-1}}$ since the determinant of $\hat{\gamma}$ is equal to $1$. The following claim is important in the proof of Lemma 3.1.
\begin{claim}
If $k>0$ is the minimal word length of $\omega \in \Gamma$ with respect to the generating set $R$ and $\pm \hat{\omega} =\pm \left( {\begin{array}{cc}
 \omega_1 & \omega_2  \\
 \omega_3 & \omega_4  \\
 \end{array} } \right) \in \hat{\Gamma}$ where $\omega_j \in \mathcal{O}_S$, then $|\sigma(\omega_j)|$ is bounded by $2^{k-1}C_1^k$ for all $j$ and $\sigma \in G$.
\end{claim}
\begin{proof}
We induct on $k$. First the case $k=1$ is clear. Suppose $k\geq 2$ and the claim is true for all $i \leq k-1$. Since $k\geq 2$, $\omega$ has one of the forms $\gamma \omega'$ or $~\gamma^{-1}\omega'$ where $\omega'$ has word length $k-1$ and $\gamma\in R$. We will prove only the case $\omega=\gamma\omega'$ because the other case is similar.
If
\begin{align*}
\pm \hat{\gamma}=\pm \left( {\begin{array}{cc}
 a & b  \\
 c & d  \\
 \end{array} } \right),~
\pm \hat{\omega'}=\pm \left( {\begin{array}{cc}
 \omega'_1 & \omega'_2  \\
 \omega'_3 & \omega'_4  \\
 \end{array} } \right),
 \end{align*}
then
\begin{align*}
\pm \hat{\omega}=\pm \left( {\begin{array}{cc}
 a & b  \\
 c & d  \\
 \end{array} } \right)
\left( {\begin{array}{cc}
 \omega'_1 & \omega'_2  \\
 \omega'_3 & \omega'_4  \\
 \end{array} } \right)=
\pm \left( {\begin{array}{cc}
 a\omega'_1+b\omega'_3 & a\omega'_2+b\omega'_4  \\
 c\omega'_1+d\omega'_3 & c\omega'_2+d\omega'_4  \\
 \end{array} } \right).
\end{align*}
Focusing on the upper left entry of $\pm\hat{\omega}$, we have
\begin{equation}\tag{4.1}
|\sigma(a\omega'_1+b\omega'_3)|\leq |\sigma(a)||\sigma(\omega'_1)|+|\sigma(c)||\sigma(\omega'_3)|.
\end{equation}
Since
\begin{equation*}
|\sigma(\omega'_1)|,~|\sigma(\omega'_3)|\leq 2^{k-2}C_1^{k-1}
\end{equation*}
for all $\sigma \in G$ by induction, we get
\begin{equation*}
|\sigma(\omega_1)|\leq (2^{k-2}C_1^{k-1})(C_1)+(2^{k-2}C_1^{k-1})(C_1)=2^{k-1}C_1^k.
\end{equation*}
for all $\sigma \in G$. The same estimate holds for the other entries of $\pm \hat{\omega}$, proving the claim.
\end{proof}

We now prove Lemma 3.1, so let $\omega \in \Gamma $ have length at most $l$ as stated in the lemma. By Claim 4.1, $|\sigma(\pm \text{tr}\hat{\omega})|$ is bounded by $2^{c'l-1}C_1^{c'l}+2^{c'l-1}C_1^{c'l}=2^{c'l}C_1^{c'l}$ for all $\sigma \in G$. Since $R$ is a finite set, we can find a common denominator $\beta'\in \mathcal{O}_K$ of all the entries of images of elements of $R$ in $SL_2(O_S)$. That is, there exists $\beta'\in \mathcal{O}_K$ such that, for any $\gamma \in R$, $\pm\hat{\gamma}$ can be represented in the following form
\begin{equation*}
\pm \left( {\begin{array}{cc}
 \alpha_{1}/\beta' & \alpha_{2}/\beta'  \\
 \alpha_{3}/\beta' & \alpha_{4}/\beta'  \\
 \end{array} } \right)
\end{equation*}
where the $\alpha_{i} \in \mathcal{O}_K$ depend on $\gamma$. If we put
\begin{equation*}
C_2=\text{max}\big\{|\sigma(\beta')|~\big|~\sigma \in G \big\},
\end{equation*}
then, as $\omega$ has word length at most $c'l$, we get $\alpha,~\beta \in \mathcal{O}_K$ such that $\pm\text{tr} \hat{\omega} =\pm\alpha/\beta$ and $|\sigma(\beta)|$ is bounded by $C_2^{c'l}$ for all $\sigma \in G$.
Since
\begin{equation*}
|\sigma(\alpha\pm 2\beta)|=|\sigma(\text{tr}\hat{\omega})\pm 2||\sigma(\beta)|,
\end{equation*}
$|\sigma(\alpha\pm 2\beta)|$ is bounded by
\begin{equation*}
(2^{c'l}C_1^{c'l}+2)C_2^{c'l}=(2C_1C_2)^{c'l}+2C_2^{c'l} \leq 2(2C_1C_2)^{c'l}
\end{equation*}
for all $\sigma \in G$. Because
\begin{equation*}
\text{N}{(\alpha\pm 2\beta)}\leq ({\text{max}\{|\sigma(\alpha\pm 2\beta)|~|~\sigma \in G\}})^m
\end{equation*}
where $m=[K:\mathbb{Q}]$, we get
\begin{equation*}
\text{N}{(\alpha\pm 2\beta)} \leq 2(2C_1C_2)^{c'lm}.
\end{equation*}
Since $l$ is the hyperbolic length of a closed manifold $M$, it is bounded below by the length $s$ of the shortest closed geodesic. Now it is straightforward to check there exists $C_3\geq 1$ such that
\begin{equation*}
(C_3)^{l}>2(2C_1C_2)^{c'lm}
\end{equation*}
for all $l\geq s$. This completes the proof of the lemma.
\section{A technical lemma for Theorem 1.3}
Now we embark on the proof of Theorem 1.3. The key idea of the proof of Theorem 1.3 is simpler than the proof of Theorem 1.2. We'll show that in the arithmetic non-compact case, the upper injectivity radius of a Hecke-type congruence subgroup is always bounded by a function of its degree.
First recall that if $M$ is an arithmetic non-compact manifold, then it has a finite cover $M'$ such that its fundamental group $\Gamma'$ is a subgroup of a \emph{Bianchi group} $PSL_2(\mathcal{O}_n)$ where $\mathcal{O}_n$ is an imaginary quadratic number field (Chapter 8 in \cite{mr}). The following lemma will provide a way to get a lower bound of upper injectivity radii of the congruence subgroups of $\Gamma'$. Although the hypotheses of the lemma may seem artificial, they are satisfied for the congruence subgroups (of $\Gamma'$) under consideration as we will explain after the proof.

\begin{lemma}
Let $\Gamma''$ be a subgroup of $PSL_2({\mathbb{C}})$ where every element that fixes $\infty \in \partial\mathbb{H}^3_{\infty}$ is parabolic. Suppose there are positive constants $C_1$ and $C_2$ so that:
\\
\\
\emph{(a)} When $\gamma=\left( {\begin{array}{cc}
 a & b  \\
 c & d  \\
 \end{array} } \right) \in \Gamma''$ does not fix $\infty$, the entry $c$ satisfies $|c| \geq C_1$.\\
\emph{(b)} When $\gamma$ is a nontrivial element which fixes $\infty$, and so of the form $\left( {\begin{array}{cc}
 1 & b  \\
 0 & 1  \\
 \end{array} } \right)$, the entry $b$ satisfies $|b| \geq C_2 $.\\
\\
Then there exists $\zeta \in \mathbb{H}^3$ such that, for every nontrivial element $\gamma \in \Gamma''$,
\begin{equation*}
\cosh{d_{\mathbb{H}^3}(\gamma(\zeta),\zeta)} \geq \dfrac{C_1C_2}{2}.
\end{equation*}
\end{lemma}

\begin{proof}
Throughout the proof, we will be working in the upper-half space model. Let $\zeta= t~\text{j}$ where $t=\left(\dfrac{C_2}{C_1}\right)^{1/2}$ and $\text{j}$ represents the vertical axis. Then it will be shown that this has the desired property. By well known formulas \cite{mat}, for $\gamma=\left( {\begin{array}{cc}
 a & b  \\
 c & d  \\
 \end{array} } \right)$, we have
\begin{equation*}
\gamma(\zeta)=\dfrac {{b\overline{d}+a\overline{c}t^2+t~\text{j}}}{|c\zeta +d|^{2}}
\end{equation*}
and
\allowdisplaybreaks\begin{align}
&\cosh {d_{\mathbb{H}^3}(\gamma(\zeta),\zeta)} =\frac {|\gamma(\zeta)-\zeta|^{2}}{\frac {2t^{2}}{|c\zeta +d|^{2}}}+1 \notag\\
&=\frac{|\gamma(\zeta)-\zeta|^{2}|c\zeta +d|^{2}}{2t^{2}}+1 \notag\\
&=\frac{|(b\overline{d}+a\overline{c}t^{2}+t~\text{j})-t~\text{j}|c\zeta +d|^{2}|^{2}}{2t^{2}|c\zeta +d|^{2}}+1\notag\\
&=\frac{|b\overline{d}+a\overline{c}t^{2}+t~\text{j}(1-|c\zeta +d|^{2})|^{2}}{2t^{2}|c\zeta +d|^{2}}+1\notag\\
&=\frac{|b\overline{d}+a\overline{c}t^{2}|^{2}}{2t^{2}|c\zeta +d|^{2}}+\dfrac{(1-|c\zeta +d|^{2})^{2}}{2|c\zeta +d|^{2}}+1.
\end{align}

First, consider the case (a). Since $c\zeta=ct~\text{j}$, we get $|c\zeta+d|^{2}\geq t^{2}|c|^{2}$. Clearly (5.1) is bigger than $\dfrac{|c\zeta +d|^{2}}{2}$ so that it is bounded below by $\dfrac{1}{2}t^{2}C_1^{2}=\dfrac{C_1C_2}{2}$ and $\cosh {d_{\mathbb{H}^3}(\gamma(\zeta),\zeta)}$ is as well.

Second, consider the case (b). In this case, we can rewrite $\cosh {d_{\mathbb{H}^3}(\gamma(\zeta),\zeta)}$ as
\begin{equation*}
\frac{b^2}{2t^{2}}+1.
\end{equation*}
Obviously this is bounded below by $\dfrac{C_2^2}{{2t^{2}}}=\dfrac{C_1C_2}{2}$. The lemma is proved.
\end{proof}
$\quad$

Now we apply the above lemma to congruence subgroups of $\Gamma'$. First, let $\Gamma'_0(I)$ be a Hecke-type congruence subgroup of $\Gamma'$ and $\gamma=\left( {\begin{array}{cc}
 a & b  \\
 c & d  \\
 \end{array} } \right) \in \Gamma'_0(I)$ where $I$ is an ideal of $\mathcal{O}_n$. By the definition of the Hecke-type congruence subgroup we have $c\in I$ and so $\text{N}(c)=|c|^2 \geq \text{N}(I)$ (when $c\neq 0$). Clearly $|b|\geq 1$ if $b\neq0$. Furthermore, if $\gamma$ fixes $\infty$ (so $c=0$), then $|a+d|\leq 2$ since $a$ and $d$ are conjugates of each other and both are units in the ring of integers in an imaginary quadratic number field. Therefore $\Gamma'_0(I)$ satisfies all the conditions of the lemma, thus, there exists $\zeta \in \mathbb{H}^3$ such that $\cosh {d_{\mathbb{H}^3}(\gamma(\zeta),\zeta)}$ is bounded below by $\dfrac{1}{2} \text{N}(I)^{1/2}$ for all nontrivial $\gamma \in \Gamma'_0(I)$. By similar method, for a Hecke-type congruence subgroup $\Gamma'_1(I)$ (resp. principal congruence subgroup $\Gamma'(I)$), we get $\cosh {d_{\mathbb{H}^3}(\gamma(\zeta),\zeta)}$ is bounded below by $\dfrac{1}{2} \text{N}(I)^{1/2}$ (resp. $\dfrac{1}{2} \text{N}(I)$) for all nontrivial $\gamma \in \Gamma'_1(I)$ (resp. $\Gamma'(I)$).
$\quad$
\section{Proof of Theorem 1.3}
If $P$ is a prime ideal of $\mathcal{O}_n$ and $M'_0(P)$ is a congruence cover of $M'$ induced by a Hecke-type congruence subgroup $\Gamma'_0(P)$ of $\Gamma'$, then the upper injectivity radius $\overline{\text{inj}}(M'_0(P))$ is bigger than
\begin{equation}
\dfrac{1}{2}\cosh ^{-1}\left(\dfrac{1}{2}\text{N}(P)^{1/2}\right)
\end{equation}
by Lemma 5.1 and the observation at the end of Section 5. From
\begin{equation*}
\cosh^2 x =\dfrac{1}{2}(\cosh (2x)+1),
\end{equation*}
we get
\begin{equation*}
\cosh\left(\dfrac{1}{2}\cosh ^{-1}(x)\right)=\sqrt{\dfrac{x+1}{2}} > \sqrt{\dfrac{x}{2}}
\end{equation*}
for $x\geq 1$. Combining this with Corollary 2.2 and (6.1), we have
\begin{equation}
\text{Heegaard genus of}~M'_0(P) \geq \dfrac{\text{N}(P)^{1/4}}{4}.
\end{equation}
Assuming $\text{N}(I)$ is sufficiently large, then the map in (2.8) is surjective and so, by the formula given in (2.9), the degree
\begin{equation*}
[\Gamma:\Gamma'_0(P)]=[\Gamma:\Gamma'][\Gamma':\Gamma'_0(P)]
\end{equation*}
is equal to $d(\text{N}(P)+1)$ where $d=[\Gamma:\Gamma']$. Furthermore, for any $\epsilon>0$, we can easily check that
\begin{equation*}
\dfrac{\text{N}(P)^{1/4}}{4} \geq (d(\text{N}(P)+1))^{\tfrac{1}{4}-\epsilon}
\end{equation*}
holds for sufficiently large $\text{N}(P)$.
This concludes
\begin{equation*}
\text{Heegaard genus of }M'_0(P)\geq ([\Gamma:\Gamma'_0(P)])^{\tfrac{1}{4}-\epsilon}
\end{equation*}
with sufficiently large $\text{N}(P)$.

For a given $\epsilon>0$, let $P_1,P_2,...$ be a sequence of prime ideals of $\mathcal{O}_n$ such that each $M'_0(P_i)$ satisfies the above condition $\dfrac{\text{N}(P_i)^{1/4}}{4} \geq (d(\text{N}(P_i)+1))^{\tfrac{1}{4}-\epsilon}$. If we put $\Gamma_i=\Gamma'_0(P_1...P_i)$ and $M_i=$ cover of $M'$ induced by $\Gamma_i$,  then
\begin{equation*}
\text{Heegaard genus of }M_i \geq \dfrac{\text{N}(P_1..P_i)^{1/4}}{4}
\end{equation*}
by the same method we used to get (6.2). Since
\begin{equation*}
\dfrac{\text{N}(P_1..P_i)^{1/4}}{4} \geq (d(\text{N}(P_1)+1)...(\text{N}(P_i)+1))^{\tfrac{1}{4}-\epsilon}
\end{equation*}
by assumption, the inequality
\begin{equation*}
\text{Heegaard genus of }M_i \geq (d(\text{N}(P_1)+1)...(\text{N}(P_i)+1))^{\tfrac{1}{4}-\epsilon}
\end{equation*}
follows for all $i$. Now the sequence $\{\Gamma_i\}$ is the desired one for Theorem 1.3.
\section{Proof of Theorem 1.5}
Throughout the proof, we shall suppose that $I$ is a square-free ideal that is not divisible by any of the prime ideals for which the map in (2.8) is not surjective. Under this assumption, we apply the explicit formulas (2.3), (2.9) and (2.10).\\
(i) As we saw in the proof of Theorem 1.3, the cover $M'_0(I)$ contains a ball $B$ of radius bigger than or equal to $\dfrac{1}{2}\cosh ^{-1}\left(\dfrac{1}{2}\text{N}(I)^{1/2}\right)$ $>0$. Since
\begin{equation*}
\cosh^{-1}\dfrac{x}{2}=\ln \left(\dfrac{x+\sqrt{x^2-4}}{2}\right)
\end{equation*}
when $x\geq 2$, $\cosh^{-1}\dfrac{x}{2}\geq0$, we get
\begin{equation*}
\dfrac{1}{2}\cosh ^{-1}\left(\dfrac{1}{2}\text{N}(I)^{1/2}\right)=\dfrac{1}{2}\ln \left(\dfrac{\text{N}(I)^{1/2}+\sqrt{\text{N}(I)-4}}{2}\right)>\dfrac{1}{2}\ln\left(\text{N}(I)^{1/2}-1\right)
\end{equation*}
for $\text{N}(I)\geq 4$. The volume of hyperbolic ball of radius $r$ is
\begin{equation*}
\pi \left(\sinh (2r)-2r\right),
\end{equation*}
so, for $r$ sufficiently large, it is bounded below by a constant multiple of $e^{2r}$. This means the volume of $B$ is bigger than a constant multiple of $e^{\ln\left(\text{N}(I)^{1/2}-1\right)}=\text{N}(I)^{1/2}-1$ with sufficiently large $\text{N}(I)$. For convenience, we will simply assume that the volume of $B$ is bounded below by the constant multiple of $\text{N}(I)^{1/2}$.

Let $I=P_1P_2...P_s$ such that $P_i$ are distinct prime ideals and $\text{N}(P_i)=p_i^{n_i}$ where $p_i$ rational primes and $n_i=1$ or $2$ depending on $P_i$. By the formula in (2.9), the degree of $M'_0$ is equal to $(p_1^{n_1}+1)(p_2^{n_2}+1)...(p_s^{n_s}+1)$. Now, for any $\epsilon>0$ and $c>0$, the following inequality
\begin{equation*}
c(p_1^{n_1}p_2^{n_2}...p_s^{n_s})^{1/2}>\left((p_1^{n_1}+1)(p_2^{n_2}+1)...(p_s^{n_s}+1)\right)^{\tfrac{1}{2}-\epsilon}
\end{equation*}
holds for sufficiently large $\text{N}(I)=p_1^{n_1}p_2^{n_2}...p_s^{n_s}$, because
\begin{equation*}
\dfrac{(p_i^{n_i})^{1/2}}{(p_i^{n_i}+1)^{\tfrac{1}{2}-\epsilon}}=(p_i^{n_i}+1)^{\epsilon}\left(\dfrac{p_i^{n_i}}{p_i^{n_i}+1}\right)^{1/2}
\end{equation*}
goes to infinity as $p_i$ increases. (i) is proved.
\\
\\
(ii) Because $\Gamma'_1(I) \subset \Gamma_0'(I)$, the cover $M'_1(I)$ (induces by $\Gamma'_1(I)$) also has a ball $B$ of radius at least $\dfrac{1}{2}\ln (\text{N}(I)^{1/2}-1)$ for $\text{N}(I)\geq 4$. As we checked in (i), the volume of this ball is bigger than the constant multiple of $\text{N}(I)^{1/2}$ for sufficiently large $\text{N}(I)$. Since the degree of $\Gamma'_1(I)$ is less than $\dfrac{1}{2}\text{N}(I)^2$, the statement in (ii) follows.
\\
\\
(iii) By the discussion at the end of Section 5,
\begin{equation*}
\cosh{d_{\mathbb{H}^3}(\gamma(\zeta),\zeta)} \geq \dfrac{1}{2}\text{N}(I)
\end{equation*}
for a principal congruence subgroup $\Gamma'(I)$ and any nontrivial $\gamma \in \Gamma'(I)$. Following the same way in (i), we can prove that the cover $M'(I)$ (induced by $\Gamma'(I)$) contains a ball $B$ of volume bounded below by the constant multiple of $\text{N}(I)$. Because the degree of $M'(I)$ is less than $\dfrac{1}{2}\text{N}(I)^3$ we arrive at the desired conclusion.
\section{Final Comments}
$\quad$\\
(1) In the proof of Theorem 1.2, we picked the prime number $p$ using Lemma 3.2. But we can choose a different prime directly from Lemma 3.1. By Lemma 3.1 (with the same notations in Section 3), for every $\omega \in \Gamma$ of length less than equal to $l$, there exist $\alpha, ~\beta \in \mathcal{O}_K$ such that $\text{tr}\hat{\omega} =\alpha/\beta$ and $\text{N} (\alpha \pm 2\beta) \leq C_3^l$. If we select a prime $p_1$ which is bigger than $C_3^l$ and smaller than $2C_3^l$, then, for a prime ideal factor $P_1$ of $p_1\mathcal{O}_K$, a Hecke-type congruence subgroup $\Gamma_1(P_1)$ does not contain any element of length less than or equal to $l$ because $\text{N}(P_1)\nmid\prod_{i=1}^{r(l)}{|\text{N}{(\alpha_i-2\beta_i)}||\text{N}{(\alpha_i+2\beta_i)}|}$ and the same reasoning in the proof of Claim 3.3. Now applying the Chebotarev's Density Theorem, pick a prime number $p'_1$ such that $p'_1$ splits completely in $\mathcal{O}_K$ and $p_1<p'_1<3p_1$. Since $p'_1$ also satisfies $p'_1\nmid \prod_{i=1}^{r(l)}{|\text{N}{(\alpha_i-2\beta_i)}||\text{N}{(\alpha_i+2\beta_i)}|}$, for any prime factor $P'_1$ of $p'_1\mathcal{O}_K$, a Hecke-type congruence subgroup $\Gamma_1(P'_1)$ does not contain any element of length less than or equal to $l$. Hence we get
\begin{equation*}
\text{Heegaard genus of the cover corresponding to } \Gamma_1(P'_1) \geq \dfrac{e^{l/2}}{4}
\end{equation*}
and, from (2.10),
\begin{equation*}
\text{Degree of }\Gamma_1(P'_1) <\dfrac{1}{2}{(p'_1)}^2< \dfrac{1}{2}(6C_3^l)^2
\end{equation*}
because $p'_1<3p_1$, $p_1<2C_3^{cl}$ and $\text{N}(P'_1)=p'_1$.

However the problem in this case is that we don't know exactly how big the constant $C_3$ is. In particular, the constant $C_3$ strongly depends on $\Gamma$. Thus, the result coming from the above line of reasoning is weaker than the result obtained using Lemma 3.2 which is universally independent of $\Gamma$.
\\
\\
(2) The reader might wonder why we chose to work with the Hecke-type subgroup $\Gamma_1(P)$ instead of the principal congruence subgroup $\Gamma(P)$ in Theorem 1.2. In fact, using $\Gamma(P)$ gives $\text{Degree}^{1/12-\epsilon}$ as a lower bound of the Heegaard genus of the induced cover. Although, for a given $l$, $\Gamma(P)$ allows us to take a smaller upper bound on $\text{N}(P)$, it doesn't offset the increase of the degree. More precisely, if a hyperbolic element $\omega$ is contained in $\Gamma(P)$, then we have $\text{tr}\hat{\omega}\equiv \pm 2\in P^2$ (compare to the case $\text{tr}\hat{\omega} \equiv \pm 2 \in P$ when $\omega \in \Gamma_1(P)$) so that we can pick a rational prime $p$ with a loosened condition $p^2\nmid \prod_{i=1}^{r(l)}|{\text{N}{(\alpha_i-2\beta_i)}}||\text{N}{(\alpha_i+2\beta_i)}|$ rather than $p\nmid \prod_{i=1}^{r(l)}|{\text{N}{(\alpha_i-2\beta_i)}}||\text{N}{(\alpha_i+2\beta_i)}|$ in the argument after Claim 3.3. By slightly changing the proof of Corollary 2.5, it is not difficult to see that for sufficiently large $x$ there exist a prime number $p\in S_2$ such that $p^2\nmid x$ and $p<n \log x$ (compare to the case $2n\log x$ of $\Gamma_1(P)$). But, as we can check from the proof of Theorem 1.2, $n \log x$ doesn't improve the result that much. On the contrary, since the degree of $\Gamma(P)$ has a cube power of $\text{N}(P)$ as one of its term, the lower bound of the Heegaard genus of the cover decreases from $\text{Degree}^{1/8-\epsilon}$ to $\text{Degree}^{1/12-\epsilon}$ .
\\
\\
(3) Now we heuristically explore the limits of Theorem 2.1. For a given closed hyperbolic 3-manifold $M$, first find the maximum upper injectivity radius of $M$. If $r=\overline{\text{inj}}(M)$ and $B(r)$ is a hyperbolic ball of radius $r$ embedded in $M$, then by assuming
\begin{equation*}
 \text{vol}(M)=\text{vol}(B(r))
\end{equation*}
one can calculate the largest possible value of $r$. Since
\begin{equation*}
\text{vol}(M)=\text{vol}(B(r)) \approx \pi e^{2r},
\end{equation*}
for sufficiently large $\text{vol}(M)$, it follows that
\begin{equation*}
 r \approx \dfrac{1}{2} \ln \dfrac{\text{vol}(M)}{\pi}.
\end{equation*}
For the convenience of calculation, we simply assume $r=\dfrac{1}{2} \ln \dfrac{\text{vol}(M)}{\pi}$. Applying this value of $r$ to Theorem 2.1, we have
\begin{equation*}
\text{Heegaard genus of } M \geq \dfrac{1}{4\sqrt{\pi}}(\text{vol}(M))^{1/2}.
\end{equation*}

In conclusion, $1/2$ is the largest value for the exponent of $\text{vol} (M)$ in Theorem 1.2 that we can get using Theorem 2.1. Recalling Theorem 1.1 we can say that Theorem 2.1 would have to be improved substantially in the arithmetic case to give an alternative proof of Theorem 1.1.

\section{Acknowledgement}
I am extremely grateful to my advisor Nathan Dunfield for all his help. This paper woudn't have been possible without his support and encouragement. 
I am also very thankful to Paul Pollack and Jonah Sinick for many helpful comments. Lastly, I would like to extend my gratitude to the anonymous referee for careful reading and corrections.

\vspace{10 mm}

Department of Mathematics\

University of Illinois at Urbana-Champaign\

1409 W. Green Street, Urbana, IL 61801\\

\emph{Email Address}: jeon14@illinois.edu

\end{document}